\documentclass[8pt]{article}
\usepackage{indentfirst,latexsym,bm}
\usepackage{amsfonts}
\usepackage{amssymb}
\usepackage[leqno]{amsmath}
\usepackage{dsfont}
\usepackage{amsthm}
\usepackage[all]{xy}
\usepackage{hyperref}
\usepackage{color,xcolor}
\begin{document}
\title{Group-theoretical property of non-degenerate fusion categories of FP-dimensions $p^2q^3$ and $p^3q^3$}
\author{Zhiqiang Yu}
\date{}
\maketitle

\newtheorem{theo}{Theorem}[section]
\newtheorem{prop}[theo]{Proposition}
\newtheorem{lemm}[theo]{Lemma}
\newtheorem{coro}[theo]{Corollary}
\theoremstyle{definition}
\newtheorem{defi}[theo]{Definition}
\newtheorem{exam}[theo]{Example}
\newtheorem{rema}[theo]{Remark}
\newtheorem{conj}[theo]{Conjecture}

\newcommand{\A }{\mathcal{A}}
\newcommand{\C }{\mathcal{C}}
\newcommand{\B }{\mathcal{B}}
\newcommand{\D }{\mathcal{D}}
\newcommand{\E }{\mathcal{E}}
\newcommand{\K}{\mathds{k}}
\newcommand{\I }{\mathcal{I}}
\newcommand{\M }{\mathcal{M}}
\newcommand{\Q }{\mathcal{O}}
\newcommand{\Y }{\mathcal{Z}}
\newcommand{\Z }{\mathbb{Z}}

\abstract
In this paper, we show that  non-degenerate fusion categories of FP-dimensions $p^2q^3d$ and $p^3q^3d$ are group-theoretical, where $p, q$ are odd primes, $d$ is a square-free integer such that $(pq,d)=1$.


\section{Introduction}
Throughout this paper, we always assume $\K$ is an algebraically closed field of characteristic zero, $\K^{*}:=\K\backslash {\{0}\}$, $\mathbb{Z}_r:=\mathbb{Z}/r$, $r\in \mathbb{N}$.  For any semisimple $\K$-linear abelian  category $\C$, we use $\Q(\C)$ to denote set of isomorphism classes of simple objects of $\C$,  cardinal of $\Q(\C)$ is denoted by $rank(\C)$.

Fusion category $\C$ is group-theoretical if it is Morita equivalent to a pointed fusion category $Vec_G^\omega$, where $G$ is a finite group, $\omega\in Z^3(G,\K^*)$ is a $3$-cocycle. Equivalently, there exists braided equivalence between Drinfeld centers $\Y(\C)\cong\Y(Vec_G^\omega)$ \cite[Theorem 1.3]{ENO3}. Since braided fusion categories $\C$ are subcategories of $\Y(\C)$, braided group-theoretical fusion categories are studied explicitly in \cite{NNW}.

 A braided fusion category $\C$ is non-degenerate, if M\"{u}ger center $\C'=Vec$.
It was proved in \cite[Theorem 4.6]{BGHKNNPR} that non-degenerate fusion categories of FP-dimension $p^2q^2$ are group-theoretical, where $p,q$ are odd primes.
In this paper, for  odd primes $p$  and $q$, we continue to   classify non-degenerate fusion categories of FP-dimension $p^2q^3$ and $p^3q^3$. We show that they are also group-theoretical fusion categories, see Theorem \ref{maintheorem2} and Theorem \ref{maintheorem3}. Consequently, combining with result of \cite[Theorem 3.2.2]{OY}, we can extend group-theoretical property to non-degenerate fusion categories of FP-dimension $p^2q^3d$ and $p^3q^3d$, where $d$ is a square-free integer such that $(pq,d)=1$.

The organization of this paper is  as follows. In section \ref{section2}, we recall some basic properties of fusion categories and braided fusion categories that we use throughout. In section \ref{section3}, we prove our two main theorems: Theorem \ref{maintheorem2} and Theorem \ref{maintheorem3}.

\section{Preliminaries}\label{section2}
In this paper, we will freely use the basic theory of fusion categories  but will recall some most used facts below, we refer the reader to \cite{DrGNO1,DrGNO2, ENO1,ENO2,ENO3,GN} for properties of fusion categories and braided fusion categories.

Let $G$ be a finite group,   $\C$ is a $G$-graded fusion category $\C=\oplus_{g\in G}\C_g$, if $\C_g\otimes\C_h\subseteq\C_{gh}$ and $(\C_g)^*=\C_{g^{-1}}$. Note that $\C_e$ is a fusion subcategory of $\C$, so we also say that $\C$ is a $G$-extension of $\C_e$. Let  $\C_{ad}$ be the  adjoint fusion subcategory of $\C$, that is, $\C_{ad}$  is  generated by simple objects $Y$ such that $Y\subseteq X\otimes X^*$ for some simple object $X$, then $\C$ has a faithful grading with $\C_{ad}$ be the trivial component \cite[Corollary 3.7]{GN}, moreover for any  grading of $\C=\oplus_{g\in G}\C_g$, we have $\C_{ad}\subseteq\C_{e}$.

For any fusion category $\C$, there exists a unique ring homomorphism FPdim(-) from Grothendieck ring $Gr(\C)$ to $\K$, which satisfies $FPdim(X)\geq1$ for all $X\in\Q(\C)$ \cite[Theorem 8.6]{ENO1}, and $FPdim(\C):=\sum_{X\in\Q(\C)}FPdim(X)^2$. Then a fusion category is pointed, if all simple objects are invertible, equivalently simple objects have FP-dimension $1$. Moreover, in this case,  $\C$ is equivalent to the category of $G$-graded vector spaces $Vec_G^\omega$, where  $\Q(\C)\cong G$ is a finite group, $\omega\in Z^3(G,\K^*)$ is a $3$-cocylce. For a fusion category $\C$, let $\C_{pt}$ be the maximal pointed fusion subcategory of $\C$, $G(\C)$ be the group of isomorphism classes of invertible objects of $\C$.

We say that a fusion category $\C$ is nilpotent, if there exists  a natural number $n$ such that $\C^{(n)}=Vec$, where $Vec$ is the semisimple category of finite-dimensional vector space over $\K$, $\C^{(0)}:=\C$, $\C^{(1)}:=\C_{ad}$, $\C^{(m)}:=(\C^{(m-1)})_{ad}$, for all $m\geq1$ \cite{GN}. And fusion category $\C$ is said to be cyclically nilpotent \cite{ENO3}, if there is a sequence of fusion categories $ Vec=\C_0\subseteq\C_1\subseteq\cdots\subseteq\C_n=\C$ and cyclic groups $G_i$ such that $\C_i$ is obtained by $G_i$-extension of $\C_{i-1}$, $1\leq i\leq n$. Obviously, pointed fusion categories are nilpotent, and fusion categories of prime power FP-dimensions are always cyclically nilpotent by \cite [Theorem 8.28]{ENO1}.  Fusion category $\C$ is group-theoretical if it is Morita equivalent to a pointed  fusion category \cite{ENO1}. Meanwhile, fusion category $\C$ is solvable if it is Morita equivalent to a cyclically nilpotent fusion category $\D$ \cite{ENO3}.

Fusion category $\C$ is a braided  fusion category, if for  arbitrary objects $X,Y,Z\in\C$, there exists a natural isomorphism $c_{X,Y}:X\otimes Y\to Y\otimes X$, which satisfies $c_{X,I}=c_{I,X}=id_X$, $$c_{X\otimes Y,Z}=c_{X,Z}\otimes id_Y \circ id_X\otimes c_{Y,Z}, \quad c_{Z,X\otimes Y}=id_X\otimes c_{Z,Y}\circ c_{Z,X}\otimes id_Y,$$
here we suppress the associativity isomorphism. Let $\D\subseteq\C$ be a braided fusion subcategory, the   centralizer  $\D'$ of $\D$ is the fusion subcategory generated by all simple objects $X$ of $\C$ such that $c_{Y,X} c_{X,Y}=id_{X\otimes Y}$, $\forall Y\in\D$. Braided fusion category $\C$ is symmetric, if $\C'=\C$; $\C$ is non-degenerate if and only if $\C'=Vec$ by \cite[Proposition 3.7]{DrGNO2}.

Symmetric fusion category $\E$ is Tannakian, if $\E\cong Rep(G)$, where $G$ is a finite group, braiding of $Rep(G)$ is given by standard reflection
$$c_{X,Y}(x\otimes y)=y\otimes x,~ x\in X,~  y\in Y,~ X, Y\in Rep(G).$$
Particularly, symmetric fusion categories of odd FP-dimensional are Tannakian categories \cite[Corollary 2.50]{DrGNO2}. Let $\E=Rep(G)$ be a  Tannakian subcategory of braided fusion category $\C$. Then the de-equivariantization $\C_G$ is a $G$-crossed braided fusion category,  and we have a fusion category equivalence $\C\cong(\C_G)^G$, see \cite[section 4]{DrGNO2} for details. In addition, $FPdim(\C_G)=\frac{FPdim(\C)}{|G|}$ and $FPdim(\C^G)= FPdim(\C)|G|$ \cite[Proposition 4.26]{DrGNO2}. If $Rep(G)=\E\subseteq\C'$, then $\C_G$ is  a braided fusion category; in particular, if $\C$ is non-degenerate, then $\C_G^0\cong\E_G'$ is also a non-degenerate fusion category and $G$-grading of $\C_G$ is faithful \cite[Proposition 4.56]{DrGNO2}. Hence, fusion category $\C$ is  group-theoretical  if and only if Drinfeld center $\Y(\C)$ contains a Lagrangian subcategory \cite[Therem 4.64]{DrGNO2}, i.e. a Tannakian subcategory $\E\subseteq\Y(\C)$ such that $\E'=\E$.

It is well-known that there is a bijective correspondence between pointed non-degenerate fusion categories and metric groups, see \cite[subsection 2.11]{DrGNO2} for details. We use $\C(G,\eta)$ to denote the non-degenerate pointed fusion category determined by metric group $(G,\eta)$, where $\eta:G\to\K^*$ is a non-degenerate quadratic form on $G$.

\section{Main result}\label{section3}
In this section,  for odd primes $p,q$, we prove that non-degenerate fusion category of FP-dimensions $p^2q^3$ and $p^3q^3$ are group-theoretical.

It was shown in \cite[Theorem 1.6]{ENO3} that fusion categories $\C$ of FP-dimensions $p^aq^b$ are solvable, where $p,q$ are primes, $a,b$ are non-negative integers. Particularly, $\C_{pt}$ is non-trivial by  \cite[Proposition 4.5]{ENO3}. Moreover, if $\D$ is a non-degenerate fusion category, then \cite[Theorem 3.14]{DrGNO2} and \cite[Corollary 6.8]{GN} imply that
$$FPdim(\D_{ad})FPdim(\D_{pt})=FPdim(\D), ~(\D_{ad})'=\D_{pt},  ~\D_{ad} =(\D_{pt})'.$$

For a  braided  fusion category $\C$, let $Aut^\text{br}_\otimes(\C)$ be the set of braided equivalences of $\C$. If $F\in Aut^\text{br}_\otimes(\C)$,   $\Q(\C)^F\subseteq\Q(\C)$ denotes the subset of simple objects of $\C$ that are fixed by $F$. The following theorem strengthens \cite[Theorem 4.6]{BGHKNNPR}.
\begin{theo}\label{maintheorem1}Let $\C$ be a non-degenerate fusion category of FP-dimensions $p^2q^2$, where $q<p$ are odd primes. Then $\C$ is pointed or $\C$ is braided equivalent to   Drinfeld center of a pointed fusion category.
\end{theo}
\begin{proof} By \cite[Theorem 4.6]{BGHKNNPR}, $\C$ is a group-theoretical fusion category. Assume $\C$ is not pointed,  then  the proof of \cite[Theorem 4.2]{BGHKNNPR}  implies that  $FPdim(\C_{ad})=p^2q$ and $FPdim(\C_{pt})=q$. Moreover $\C_{pt}\subseteq\C_{ad}$ is a Tannakian subcategory by \cite[Corollary 6.8]{GN}.

Let $\E=Rep(G)$ be a maximal Tannakian subcategory of $\C$, since $|G|^2|FPdim(\C)$ \cite[Theorem 3.14]{DrGNO2}, $|G|\in{\{q,p,pq}\}$. We show $|G|=pq$. Indeed, $\C_{pt}$ is   Tannakian, so $\C_{pt}$ is contained in a maximal Tannakian subcategory,  while \cite[Theorem 5.9]{DrGNO2} shows that maximal Tannakian subcategories of $\C$ have same FP-dimension, so $|G|\neq p$.  If $|G|=q$, so $\C_{pt}$ is a maximal  Tannakian subcategory.  Note that $\C$ is group-theoretical but not pointed, it follows from \cite[Theorem 7.3]{NNW} that there exists a Tannakian subcategory $Rep(N)$ such that $\C_N$ is pointed. Let $\D:=\C_N=\oplus_{g\in N}\D_g$ with $\C_N^0=\D_e=\C(H,\eta)$ be a non-degenerate pointed fusion category of FP-dimension $p^2$. Therefore, $\forall g\in N$, $p^2=rank(\D_g)=|\Q(\D_e)^g|$ by \cite[Lemma 10.7]{Kir},  which means that $N$ acts trivially on $\C(H,\eta)$. However, we have $(q,|H|)=1$, then $\E'\cong\C(H,\eta)^N\cong\C(H,\eta)\boxtimes Rep(N)$ is a pointed fusion category, this contradicts to $FPdim(\C_{pt})=q$.

Hence $FPdim(\E)=pq$, $\E$ is a Lagrangian subcategory of $\C$ by definition.  It follows from \cite[Theorem 4.64]{DrGNO2} that $\C\cong \Y(Vec_G^\omega)$,  where   $\omega\in Z^3(G,\K^*)$ is a $3$-cocycle.
\end{proof}
Next we classify non-degenerate fusion categories of FP-dimension $p^2q^3$.
\begin{lemm}\label{lemma}Braided fusion category $\A$ of FP-dimension $pq$ are symmetric or pointed, where $p<q$ are odd primes.
\end{lemm}
\begin{proof}Assume that $\A$ is not symmetric, then we consider M\"{u}ger center $\A'$. If $\A'=Vec$, then $\A$ is a  non-degenerate fusion category, as $FPdim(\C)$ is an odd square-free integer, $\C$ must be a pointed fusion category by \cite[Theorem 2.11]{ENO3}. If $\A'\neq Vec$, then $\A'$ is a  Tannakian fusion category of prime FP-dimension, since $\A$ is not symmetric. Therefore, $\A\cong\C(\Z_p,\eta_1)^{\Z_q}$ or $\A\cong\C(\Z_q,\eta_2)^{\Z_p}$ as braided fusion categories by \cite[Corollary 4.31]{DrGNO2}. As  $p\neq q$ are   odd primes, we obtain that there is an  equivalence $\A\cong\C(\Z_p,\eta_1)\boxtimes Rep(\Z_q)$ or $\A\cong\C(\Z_q,\eta_2)\boxtimes Rep(\Z_p)$, hence $\A$ is pointed.
\end{proof}
For a braided fusion category $\A$, let $\B,\D$ be its fusion subcategories, we use $\B\vee\D$ to denote the fusion subcategory generated by $\B$ and $\D$. Let $\C$ be a non-degenerate fusion category of FP-dimensions $p^2q^3$, then   \cite[Proposition 8.15]{ENO1} says that $$FPdim(\C_{ad})\in{\{1,p,q,p^2,q^2,pq,p^2q,pq^2,q^3,p^2q^2,pq^3,p^2q^3}\}.$$
\begin{theo}\label{maintheorem2}Let  p and q be odd primes, then non-degenerate fusion categories of FP-dimension $p^2q^3$ are group-theoretical.
\end{theo}
\begin{proof}
Let $\C$ be a non-degenerate fusion category of FP-dimensions $p^2q^3$. As $ \C_{pt}$ is non-trivial by \cite[Proposition 4.5]{ENO3}, \cite[Theorem 3.14]{DrGNO2} and \cite[Corollary 6.8]{GN} together imply that  $FPdim(\C_{ad})\neq p^2q^3$. Assume that $\C$ is not pointed.  In particular, $\C$ is not nilpotent, since if $\C$ is nilpotent, then \cite[Theorem 6.10]{DrGNO1} shows that $\C$ is braided equivalent to Deligne tensor product of non-degenerate fusion categories of FP-dimensions $p^2$ and $q^3$, which are pointed fusion categories, contradiction. Then $$FPdim(\C_{ad})\notin{\{p,q,p^2,q^2,pq,q^3,p^2q^2}\}.$$
In fact, if $FPdim(\C_{ad})$ is a prime power, then  $\C$ is nilpotent \cite[Theorem 8.28]{ENO1}.  If $FPdim(\C_{ad})= pq$, then  Lemma \ref{lemma} and \cite[Corollary 6.8]{GN} mean  that $\C_{ad}\subseteq\C_{pt}$, hence $\C$ is nilpotent. If $FPdim(\C_{ad})=p^2q^2$, then $\C_{pt}\subseteq\C_{ad}$ is a  Tannakian category  of FP-dimension $q$, which implies that $\C_{ad}\cong\D^{\Z_q}$, where $\D$ is pointed non-degenerate fusion category of FP-dimension $p^2q$ by \cite[Corollary 4.31]{DrGNO2}. Then $\C_{ad}$ contains a braided fusion subcategory $\C(\Z_q,\eta)^{\Z_q}$ of FP-dimension $q^2$, which must be pointed, this is impossible.

If $FPdim(\C_{ad})= pq^3$, then $\C_{pt}\subseteq\C_{ad}$ is a  Tannakian category  of FP-dimension $p$,
and $\C_{ad}\cong\D^{\Z_p}$, where $\D$ is pointed non-degenerate fusion category of FP-dimension $q^3$ by \cite[Corollary 4.31]{DrGNO2}.
Let   $\A:=\C_{\Z_p}=\oplus_{g\in\Z_p}\A_g$ with $\A_e=\D$. If $\A$ is not pointed, then simple objects of $\A$ have FP-dimension $1$ or $q$  \cite[Corollary 5.3]{GN}, again \cite[Lemma 10.7]{Kir}  shows that $rank(\A_g)=|\Q(\D)^g|=q$ . Let $X\in\Q(\D)^g$ be a non-trivial simple object, and $\B$ be the fusion category generated by $X$. Since $\Z_p$ is cyclic group of prime order,  there is a well-defined group homomorphism from $\Z_p$ to $Aut^\text{br}_\otimes(\B)$,   \cite[Theorem 9.5]{ENO2} says that we have a braided fusion category $\B^{\Z_p}$, which is a pointed fusion category of FP-dimension $pq$ by \cite[Proposition 4.26]{DrGNO2}, this is impossible.  Hence,   $\A$ is a pointed fusion category, and it follows from \cite[Theorem 7.3]{NNW}  that $\C$ is group-theoretical.

If $FPdim(\C_{ad})= pq^2$ and $FPdim(\C_{pt})=pq$, then $(\C_{ad})_{pt}=Rep(G)$ is a  Tannakian category  of FP-dimension $p$ or $q$ or $pq$. If $FPdim((\C_{ad})_{pt})=pq$, as $\C_G$ is faithfully graded by  $G$ and $\C_G^0$ is pointed,  then  \cite[Corollary 5.3]{GN} shows that $\C_G$ must be pointed, hence $\C$ is group-theoretical by \cite[Theorem 7.3]{NNW}. If $FPdim((\C_{ad})_{pt})=q$. Then $$\C_{ad}\cong (\C(\Z_p,\eta_1)\boxtimes\C(\Z_q,\eta_1))^{\Z_q}\cong \C(\Z_p,\eta_1)^{\Z_q}\vee\C(\Z_q,\eta_1)^{\Z_q},$$
which then implies that $\C_{ad}$ is pointed, as $\C_{ad}$ is generated by two pointed fusion subcategories, this is  impossible. If $FPdim((\C_{ad})_{pt})=p$, then again we have equivalences
 $$\C_{pt}\cong \C(\Z_q,\eta)^{\Z_p}\cong\C(\Z_q,\eta)\boxtimes Rep(\Z_p),$$
hence $\C$ contains a non-degenerate fusion category $\C(\Z_q,\eta)$. Then  $\C\cong\C(\Z_q,\eta)\boxtimes\C(\Z_q,\eta)'$ by \cite[Theorem 3.13]{DrGNO2}. By Theorem \ref{maintheorem1} or \cite[Theorem 4.6]{BGHKNNPR}, we obtain that $\C$ is a group-theoretical fusion category.

If $FPdim(\C_{ad})= p^2q$, then $(\C_{ad})_{pt}$ is a  Tannakian category  of FP-dimension $q$.
Let $\E=(\C_{ad})_{pt}=Rep(G)$ with $G=\Z_q$, assume that $\D=\oplus_{g\in G}\D_g:=\C_G$ and $\D_e=\C_G^0\cong\C(H,\eta_1)\boxtimes\C(\Z_q,\eta_2)$. Since $\D_e$ is pointed, $\D$ is nilpotent, for any $X\in\Q(\D)$,  $FPdim(X)\in{\{1,p}\}$ by \cite[Corollary 5.3]{GN}. Particularly, for any $g$, $rank(\D_g)=q$ or  $p^2q$.

If  $rank(\D_g)= p^2q$ for some $g\neq e$,  obviously $\D$ must be pointed, as $rank(\D_g)=|\Q(\D_e)^g|$ by \cite[Lemma 10.7]{Kir} and $G$ is a cyclic group of prime order. Then $\C$ is group-theoretical by \cite[Theorem 7.3]{NNW}.

If $rank(\D_g)=q$ for all non-trivial element $g\in G$,  again by \cite[Lemma 10.7]{Kir}  we have that $rank(\D_g)=|\Q(\D_e)^g|$. Since $ \D_e$ is pointed and $G$ acts as braided equivalence on $\D_e$,  set $\Q(\D_e)^g$ is an abelian group of order $q$. Therefore $G$ acts non-trivially on $\C(H,\eta_1)$ with no fixed pointed.  By \cite[Theorem 9.5]{ENO2} and \cite[Theorem 4.44]{DrGNO2}, we have a non-degenerate fusion category $\A$ of FP-dimension $p^2q^2$, which contains $\C(H,\eta_1)^G$ as a subcategory. Note that $\C(H,\eta_1)^G\cong\C_{ad}$, which is not pointed, then Theorem \ref{maintheorem1} implies that $\A$ is braided equivalent Drinfeld center of a pointed fusion category.  Hence maximal Tannakian subcategories of $\A$ have FP-dimension $pq$ by Theorem \ref{maintheorem1}. Since $\E=(\C_{ad})_{pt}\subseteq\C(H,\eta_1)^G\subseteq\A$ is Tannakian, $\E$ is contained in a maximal Tannakian subcategory $\E_1$ of $\A$. Then $\E_1\subseteq\E'_\A\cong\C(H,\eta_1)^G\cong \C_{ad}$, where $\E'_\A$ is the  centralizer of $\E$ in $\A$. Therefore, $\C$ contains a Tannakian subcategory $Rep(N)$ of FP-dimension $pq$, similarly $\C_N$ is a pointed fusion category, therefore $\C$ is also a group-theoretical fusion category by \cite[Theorem 7.3]{NNW}.
\end{proof}

\begin{theo}\label{maintheorem3}Let  p and q be odd primes, then non-degenerate fusion categories of FP-dimension $p^3q^3$ are group-theoretical.
\end{theo}
\begin{proof}
First, note  that we also have
$$FPdim(\C_{ad})\in{\{1,p,q,p^2,q^2,pq,p^2q,pq^2,p^3,q^3,p^2q^2,pq^3,p^3q,p^2q^3,p^3q^2,p^3q^3}\}.$$
Then, assume $\C$ is not pointed,  as in Theorem \ref{maintheorem2},  we have $$FPdim(\C_{ad})\notin{\{1,p,q,p^2,q^2,pq,p^3,q^3,pq^3,p^3q,p^2q^3,p^3q^2,p^3q^3}\}.$$

If $FPdim(\C_{ad})=p^2q^2$, then we claim that $\C_{pt}\subseteq\C_{ad}$ is Tannakian. Indeed, if $FPdim((\C_{ad})_{pt})=p$ or $q$, then $\C_{ad}$ contains a pointed fusion category of FP-dimension $p^2$ or $q^2$ as Theorem \ref{maintheorem1}. Hence, let $\C_{pt}=Rep(G)$,  since $\C_G$ has a faithful $G$-grading with trivial component be a pointed fusion category of FP-dimension $pq$ \cite[Proposition 4.56]{DrGNO2}, we obtain that $\C_G$ nilpotent, then $\C_G$ is pointed by \cite[Corollary 5.3]{GN}. Therefore, $\C$ is a group-theoretical fusion category by \cite[Theorem 7.3]{NNW}.

If $FPdim(\C_{ad})=p^2q$ and $FPdim(\C_{pt})=pq^2$. We consider  $FPdim((\C_{ad})_{pt})$. If  $FPdim((\C_{ad})_{pt})=pq$, then $\C$ is a group-theoretical fusion category as above. If $FPdim((\C_{ad})_{pt})=p$,   similarly we obtain that  $p^2|FPdim((\C_{ad})_{pt})$, this is impossible. Therefore, $FPdim((\C_{ad})_{pt})=q$. Then $(\C_{ad})_{pt}\cong Rep(\Z_q)$ is Tannakian, and it follows from \cite[Corollary 4.31]{DrGNO2} that
$$\C_{pt}\cong(\C(\Z_p,\eta_1)\boxtimes\C(\Z_q,\eta_2))^{\Z_q}
\cong\C(\Z_p,\eta_1)^{\Z_q}\vee\C(\Z_q,\eta_2)^{\Z_q}\supseteq\C(\Z_p,\eta_1)^{\Z_q}.$$
While $\C(\Z_q,\eta_2))^{\Z_q}\cong\C(\Z_p,\eta_1)\boxtimes Rep(\Z_q)$, therefore,
$\C\cong\C(\Z_p,\eta)\boxtimes\C(\Z_p,\eta)'$ by \cite[Theorem 3.13]{DrGNO2},  and $\C(\Z_p,\eta)'$ is a non-degenerate fusion category of FP-dimension $p^2q^3$, hence $\C$ is a group-theoretical fusion category  by Theorem \ref{maintheorem2}.

When $FPdim(\C_{ad})=pq^2$ and $FPdim(\C_{pt})=p^2q$, it suffices to change $p$ with $q$ in above proof, then the rest  is same, so $\C$  is   group-theoretical.
\end{proof}

\begin{coro}Non-degenerate fusion categories   of FP-dimension $p^2q^2d$,   $p^2q^3d$ and $p^3q^3d$ are group-theoretical, where  p and q  are odd primes,  $d$ is a square-free integer such that $(pq,d)=1$.
\end{coro}
\begin{proof}Let $\C$ be such a non-degenerate fusion category. It was proved in \cite[section 3.2]{OY} that we have an equivalence of braided fusion categories $\C\cong\D\boxtimes\C(\Z_d,\eta)$, where $\D$ is a non-degenerate fusion category  of FP-dimension $p^2q^2$, $p^2q^3$ or $p^3q^3$. Then $\C$ is group-theoretical by Theorem \ref{maintheorem1},  Theorem \ref{maintheorem2} and Theorem \ref{maintheorem3}, respectively.
\end{proof}

\begin{rema}If odd primes $p,q$ satisfy $p|(q+1)$, then it follows from \cite[Theorem 1.1]{JL} that there exists fusion category $\C$ of FP-dimension $pq^2$, which is not group-theoretical. Then Drinfeld center $\Y(\C)$ is  not group-theoretical and  $FPdim(\Y(\C))=p^2q^4$.
\end{rema}

\section*{Acknowledgements}
The author is   grateful to Professor V. Ostrik   weekly insightful conversations.  This paper was written during a visit of the author at University of Oregon supported by China Scholarship Council (grant. No 201806140143), he  appreciates the Department of Mathematics for their warm hospitality.

\bigskip
\author{Zhiqiang Yu\\ \thanks{Email:\,zhiqyu-math@hotmail.com}
\\{\small College  of Mathematical Science,  East China Normal University,
Shanghai 200241, China}
}

\end{document}